\documentclass[11pt]{amsart}
\usepackage{amsmath}
\usepackage{amssymb}
\usepackage{amsthm}
\usepackage{amsfonts}
\usepackage{color}
\usepackage{hyperref}

\newtheorem{theorem}{Theorem}[section]
\newtheorem{lemma}[theorem]{Lemma}
\newtheorem{corollary}[theorem]{Corollary}

\theoremstyle{definition}
\newtheorem{definition}[theorem]{Definition}

\theoremstyle{remark}
\newtheorem{remark}[theorem]{Remark}

\numberwithin{equation}{section}

\newcommand{\m}{\mathord{-}}
\newcommand{\p}{\mathord{+}}
\newcommand{\db}{\bar{\partial}}
\newcommand{\dd}{\partial}

\title{Pluriclosed Flow and Hermitian-Symplectic Structures}
\author{Yanan Ye}
\email{\href{mailto:yeyanan@outlook.com}{yeyanan@outlook.com}}
\date{}

\begin{document}
\maketitle

\begin{abstract}
We show pluriclosed flow preserves the Hermitian-symplectic structures.
And we observe that it can actually become a flow of Hermitian-symplectic forms when an extra evolution equation determined by the Bismut-Ricci form is considered.
Moreover, we get a topological obstruction to the long-time existence in arbitrary dimension.
\end{abstract}

\section{Introduction}

A Hermitian-symplectic structure, which is introduced by Tian and Streets in \cite{PF}, is a symplectic form such that its $(1,1)$-part $\omega$ is positive.
In fact, the $(1,1)$-part $\omega$ is pluriclosed.
So we can consider its deformation along the pluriclosed flow, which is also introduced by Tian and Streets in \cite{PF}.

Formally, the pluriclosed flow is
\begin{equation*}
\left\{
\begin{aligned}
&\frac{\dd}{\dd t}\omega
=-\rho^{1,1}(\omega)
\\
&\omega(0)=\omega_0
\end{aligned}
\right.
\end{equation*}
Here, $\rho^{1,1}$ is the $(1,1)$-part of Bismut-Ricci form $\rho$ which will be defined later.

We say a pluriclosed metric $\omega$ is Hermitian-symplectic if it can be extended to a Hermitian-symplectic form as the $(1,1)$-part.
In other words, there exists a $(2,0)$ form $\varphi$ such that $\Omega\triangleq\varphi+\omega+\bar{\varphi}$ is a Hermitian-symplectic form.

Now, we can state our first theorem.
\begin{theorem}
Pluriclosed flow preserves Hermitian-symplectic structures.
\end{theorem}
That means if the initial metric $\omega_0$ is Hermitian-symplectic, then the solution $\omega(t)$ is Hermitian-symplectic at any time.

Notice that when $\omega$ is Hermitian-symplectic, the selection of $(2,0)$ form $\varphi$ is not unique.
In fact, different selections differ by a closed $(2,0)$ form.
The pluriclosed flow only gives the deformation of $(1,1)$-part $\omega$.
Although the theorem above tells us that the solution $\omega(t)$ is Hermitian-symplectic, there is not a canonical way to choose $\varphi(t)$, the smooth one-parameter family of $(2,0)$ forms.
One approach is to determine the smooth one-parameter family $\varphi(t)$ using an evolution equation.

By an observation, we find the following equation satisfies the requirements.

\begin{equation}
\label{HSF}
\left\{
\begin{aligned}
&\frac{\partial}{\partial t}\omega
=-\rho^{1,1}(\omega)
\\
&\frac{\partial}{\partial t}\varphi
=-\partial\text{\rm tr}_{g}(\bar{\partial}\varphi)
\\
&\omega(0)=\omega_0,\quad \varphi(0)=\varphi_0.
\end{aligned}
\right.
\end{equation}
For any selection of $(2,0)$ form $\varphi_0$, it is a flow of Hermitian-symplectic forms with initial data $\Omega_0=\varphi_0+\omega_0+\bar{\varphi}_0$.

The idea of flowing a $(2,0)$ form by the $(2,0)$-part of Bismut-Ricci form along the pluriclosed flow first appears in \cite{OtherR1} (can also see \cite{OtherR2, MarioStreetsBOOK}).
They consider the system consisting of pluriclosed flow and an evolution equation of $(2,0)$ form
\begin{align*}
\frac{\dd}{\dd t}\beta=-\rho^{2,0}(\omega).
\end{align*}
And direct computation shows that the evolution equation of $(2,0)$-part in \eqref{HSF} is exactly the evolution equation above under the Hermitian-symplectic assumption.
So we can rewrite \eqref{HSF} in a simple form
\begin{align*}
\frac{\dd}{\dd t}\Omega=-\rho(\Omega^{1,1}).
\end{align*}

In general, for a Hermitian-symplectic form $\Omega=\varphi+\omega+\bar{\varphi}$, we can define a positive quantity
\begin{align*}
\mathcal{V}(\Omega)\triangleq\frac{1}{(k!)^2}
\sum_{k\geq0}(\varphi^k,\varphi^k)_{\omega}.
\end{align*}
Here $(\cdot,\cdot)_{\omega}$ is the inner product induced by $\omega$.
And considering the function $\mathcal{V}(t)\triangleq\mathcal{V}(\Omega(t))$ along a solution to \eqref{HSF}, we show that it is actually a polynomial about $t$ whose coeifficients are determined only by initial data and the underlying complex manifold.

More precisely, we have the following theorem.
\begin{theorem}
For a solution to \eqref{HSF} with initial data $\Omega_0$, the function $\mathcal{V}(t)=\mathcal{V}(\Omega(t))$ is a polynomial of degree at most $n$
\begin{align*}
\mathcal{V}(t)
=\sum_{i=0}^{n}a_i t^i
\end{align*}
with coefficients
\begin{align*}
a_i
=\frac{1}{i!}
\sum_{\substack{k\geq0\\2k+i\leq n}}
\frac{1}{(k!)^2(n\m 2k\m i)!}
\int_{M}
\varphi_0^k
\wedge\bar{\varphi}_0^k
\wedge\omega^{n-2k-i}_0
\wedge(\sqrt{-1}\dd\db\log\det g_0)^{i}
\end{align*}
Specially, we have
\begin{align*}
a_n=\frac{1}{n!}(-1)^n
\int_{M}
c_{1}(M)^{n}
=\frac{1}{n!}(-1)^n c_1^n
\end{align*}
in which $c_{1}(M)$ is the first Chern class of $(M^{2n},J)$.
\end{theorem}
Thus $\mathcal{V}(t)$ gives an obstruction to how long solutions exist for $\mathcal{V}(t)$ is strictly positive.
Especially, we get a topological obstruction to global solutions, (i.e. solutions exist on $[0,+\infty)$), to pluriclosed flow with Hermitian-symplectic initial data.
\begin{theorem}
If there exists a global solution to pluriclosed flow with Hermitian-symplectic initial data on $(M^{2n},J)$, then  $(-1)^n c_{1}^{n}$ must be a nonnegative integer.
\end{theorem}

One of the most important motivation for introducing pluriclosed flow is to study non-K{\"a}hler manifolds, especially to classify Kodaira\rq{}s class VII surfaces (see \cite{Streets-Geometrization, PF}).
We hope a pluriclosed metric may deform to some canonical metric along this flow.
A natural possible canonical metric is static metric, which is introduced in \cite{PF}.
Recall that a metric is static if there exists a real number $\lambda$ such that

\begin{align*}
\rho^{1,1}(\omega)=\lambda\omega
\end{align*}
Actually, static metric is Hermitian-symplectic when $\lambda\neq0$.
Although Tian and Streets\cite{PF} have proven non-K{\"a}hler surfaces does not admit Hermitian-symplectic structures using the classification of compact complex surfaces.
But it is still not clear for high dimensional cases.

Now we can have another approach to study the Hermitian-symplectic structures, which is to flow them along \eqref{HSF}.
In the case of compact complex surfaces, we prove that the function $(\varphi(t),\varphi(t))_{\omega(t)}$ is monotonically decreasing.
As a corollary, we show again that non-K{\"a}hler surfaces do not exist static metric with $\lambda<0$ without help of classification of compact complex surfaces.

Here is an outline of the rest of this paper.
In section 2 we recall some basic notions about pluriclosed flow and Hermitian-symplectic structures.
In section 3 we show that pluriclosed flow preserves Hermitian-symplectic structures and study the flow \eqref{HSF}.
In section 4 we give a topological obstruction to global solutions using the function $\mathcal{V}(t)$ defined above.
Finally in section 5 we do some discussion in surfaces case.
\newline

\textbf{Acknowledgements.}
I would like to express my sincere gratitude to my advisor Professor Gang Tian, for his helpful suggestions and patient guidance.
And thanks to Professor Jeffrey Streets for his helpful and detailed comments.
\section{Preliminary}

In this section, we review some notions about pluriclosed flow and Hermitian-symplectic structures.

\subsection{Pluriclosed flow}
Consider a Hermitian manifold $(M^{2n},J,g)$.
Bismut\cite{Bismut} shows that there exists a unique real connection $\nabla$ compatible with complex structure and Hermitian metric such that its torsion tensor is totally skew-symmetric.
That means
\begin{equation*}
\nabla J=\nabla g=0
\end{equation*}
and
\begin{equation*}
H(X,Y,Z)=-H(Y,X,Z)=-H(X,Z,Y).
\end{equation*}
Here
\begin{equation*} 
H(X,Y,Z)=g(\nabla_{X}Y-\nabla_{Y}X-[X,Y],Z)
\end{equation*}
is the tensor induced by torsion operator.

We denote $\omega$ the fundamental form corresponding to $g$.
A metric is pluriclosed if its fundamental form is pluriclosed, i.e. $\partial\bar{\partial}\omega=0$.
The torsion tensor $H$ corresponding to Bismut connection is a real 3-form and we have the formula $H=-d^{c}\omega=d\omega(J\cdot,J\cdot,J\cdot)$, in which $d^{c}=\sqrt{-1}(\bar{\partial}-\partial)$ is the conjugate differential operator.
Then a metric $\omega$ is pluriclosed if and only if its Bismut torsion tensor is closed for $-dd^{c}\omega=-2\sqrt{-1}\partial\bar{\partial}\omega$.
Gauduchon\cite{Gauduchon} shows that every complex surface admits pluriclosed metric.\\
We denote the Riemann operator corresponding to Bismut connection by
\begin{equation*}
R(X,Y)Z=\nabla_{X}\nabla_{Y}Z-\nabla_{Y}\nabla_{X}Z-\nabla_{[X,Y]}Z.
\end{equation*}
For Bismut connection, we can also define a Ricci-type form by
\begin{equation*}
\rho(X,Y)=\sum_{i=1}^{n}g(R(X,Y)e_{i},Je_{i})
=\sqrt{-1}\sum_{i=1}^{n}g(R(X,Y)Z_{i},\bar{Z}_{i}).
\end{equation*}
Here, $\{e_i,Je_i\}$ is any local orthonormal real basis.
And $\{Z_i\}$ and $\{\bar{Z}_i\}$ are unitary bases determined by
\begin{equation*}
Z_i=\frac{\sqrt{2}}{2}(e_i-\sqrt{-1}Je_i),\quad \bar{Z}_i=\frac{\sqrt{2}}{2}(e_i+\sqrt{-1}Je_i).
\end{equation*}
It is easy to check that the definition of $\rho$ is independent of choice of basis.

Now, we are in a position to define pluriclosed flow.
\begin{equation*}
\left\{
\begin{aligned}
&\frac{\partial}{\partial t}\omega(t)=-\rho^{1,1}(\omega(t))
\\
&\omega(0)=\omega_0.
\end{aligned}
\right.
\end{equation*}
Here, $\rho^{1,1}$ is the (1,1)-part of Ricci form, that is
\begin{equation*}
\rho^{1,1}(X,Y)=\frac{1}{2}\left(\rho(X,Y)+\rho(JX,JY)\right).
\end{equation*}
Tian and Streets\cite{PF} show that this flow preserves pluriclosed condition.
That means the solution is pluriclosed at any time if the initial metric is pluriclosed.
Meanwhile, they\cite{PF,HCF,Regularity} prove it is a strictly parabolic systems under the pluriclosed assumption and give the short time existence and some basic regularity results.
More results about regularity and long-time existence can be found in \cite{OtherR2, OtherR3, OtherR1, StreetsMoreResults2, StreetsMoreResults1}.
\begin{remark}
In local coordinates, we can express pluriclosed flow by Hodge operator as
\begin{equation}
\left\{
\begin{aligned}
&\frac{\partial}{\partial t}\omega=\partial\partial^{*}\omega
+\bar{\partial}\bar{\partial}^{*}\omega
+\sqrt{-1}\partial\bar{\partial}\log\det g
\\
&\omega(0)=\omega_0.
\end{aligned}
\right.
\label{Eq:pluriclosedFlowHodge}
\end{equation}
Here, $\partial^{*}$ is the dual operator of $\partial$ corresponding to the inner product induced by metric $\omega(t)$.
\end{remark}

\subsection{Hermitian-symplectic structure}
In this subsection, we will review the Hermitian-symplectic structure introduced by Tian and Streets\cite{PF}.\\
A Hermitian-symplectic form on a complex manifold $(M^{2n},J)$ is a real symplectic form  such that its (1,1)-part is positive.
In other words, a Hermitian-symplectic form is a real 2-form $\Omega$ such that $d\Omega=0$ and in local coordinates $\{z_i\}$,
\begin{equation*}
\Omega^{1,1}=\sqrt{-1}g_{i\bar{j}}dz^{i}d\bar{z}^{j}.
\end{equation*}
in which $(g_{i\bar{j}})_{n\times n}$ is a positive definite Hermitian matrix.
Thus, the (1,1)-part of a Hermitian-symplectic form is the fundamental form corresponding to some metric.

A K{\"a}hler form is naturally a Hermitian-symplectic form.
Tian and Streets\cite{PF} prove that a complex surface admits Hermitian-symplectic structures if and only if it is K{\"a}hler. But it is not clear if is it still true for dimension higher than two.

For $\Omega$ is real and closed, we have
\begin{equation*}
\left\{
\begin{aligned}
&\bar{\partial}\Omega^{2,0}+\partial\Omega^{1,1}=0
\\
&\partial\Omega^{2,0}=0.
\end{aligned}
\right.
\end{equation*}
Here, $\Omega=\Omega^{2,0}+\Omega^{1,1}+\overline{\Omega^{2,0}}$ and $\Omega^{1,1}$ are real forms.
Applying $\bar{\partial}$ to the first equation, we get the fact that the (1,1)-part of a Hermitian-symplectic form is pluriclosed.
\\
On the other hand, given a pluriclosed metric $\omega$, consider the equations with respect to (2,0)-form $\varphi$
\begin{equation}
\left\{
\begin{aligned}
&\bar{\partial}\varphi+\partial\omega=0
\\
&\partial\varphi=0.
\end{aligned}
\right.
\label{Eq:existenceOfHS}
\end{equation}
If a (2,0)-form $\varphi$ is a solution, then $\widetilde{\omega}\triangleq\varphi+\omega+\overline{\varphi}$ is a Hermitian-symplectic form.
And in this case we call the pluriclosed form $\omega$ can be extended to a Hermitian-symplectic form.
Notice that if equations (\ref{Eq:existenceOfHS}) has a solution $\varphi$, then $\varphi+\phi$ is also a solution for any closed $(2,0)$-form $\phi$.

\section{A parabolic flow of Hermitian-symplectic forms}

In this section, we first directly prove that pluriclosed flow preserves Hermitian-symplectic structures.
Then we study the flow of Hermitian-symplectic forms \eqref{HSF}.
We will show this flow is parabolic.
And by the standard theory of parabolic equations, we get the short-time existence and uniqueness of solutions.
As a corolarry, we prove again that pluriclosed flow preserves Hermitian-symplectic structures.
Meanwhile, we will see that if the initial Hermitian-symplectic form is degenerate, i.e. its (1,1)-part is K{\"a}hler, then the flow of (1,1)-part will degenerate to K{\"a}hler-Ricci flow and the (2,0)-part is constant.

Firstly, we prove the next theorem.
\begin{theorem}
Pluriclosed flow preserves Hermitian-symplectic structures.
\end{theorem}
\begin{proof}
Assume the solution $\omega(t)$ with Hermitian-symplectic initial data exists on $[0,T)$.
Applying $\dd$ to the first equation of \eqref{Eq:pluriclosedFlowHodge}, we get a evolution equation of $\dd\omega$
\begin{align*}
\frac{\dd}{\dd t}\dd\omega
=\dd\db\db^*\omega.
\end{align*}
And we will use it to find the $(2,0)$-part $\varphi(t)$ satisfying $\partial\omega(t)=-\bar{\partial}\varphi(t)$ for all $t\in[0,T)$.
\\
Choose a fixed background metric $h$ and denote $(\cdot,\cdot)_h$ the induced inner product on space of differential forms.

For a test $(2,1)$-form $\eta$, derivative $(\partial\omega(t),\eta)_h$ with respect to $t$
\begin{align*}
\frac{d}{dt}(\partial\omega(t),\eta)_h
&=(\frac{\partial}{\partial t}\partial\omega(t),\eta)_h
\\
&=(\partial\bar{\partial}\bar{\partial}^{*_{t}}
\omega(t),\eta)_h
\\
&=(-\bar{\partial}\partial\bar{\partial}^{*_{t}}
\omega(t),\eta)_h
\\
&=(-\partial\bar{\partial}^{*_{t}}\omega(t),\bar{\partial}^{*_{h}}\eta)_h
\end{align*}
By Newton-Leibniz formula,
\begin{align*}
(\partial\omega(t),\eta)_h-(\partial\omega(0),\eta)_h
&=\int_{0}^{t}\frac{d}{ds}
(\partial\omega(s),\eta)_h ds
\\
&=\int_{0}^{t}(-\partial\bar{\partial}^{*_{s}}\omega(s),\bar{\partial}^{*_{h}}\eta)_h ds
\\
&=-(\int_{0}^{t}\partial\bar{\partial}^{*_{s}}
\omega(s) ds,\bar{\partial}^{*_{h}}\eta)_h.
\end{align*}
This last equal is because the order of integration can be exchanged.

Note $\omega_0$ can extend to a Hermitian-symplectic form.
Thus there is a $(2,0)$-form $\varphi_0$ satisfying $\partial\omega_0=-\bar{\partial}\varphi_0$ and $\partial\phi_0=0$.
\\
Then we have
\begin{align*}
(\partial\omega(t),\eta)_0
=-(\varphi_0
+\int_{0}^{t}\partial\bar{\partial}^{*_{s}}
\omega(s) ds,\bar{\partial}^{*_0}\eta)_0.
\end{align*}
Thus, we can define $\varphi(t)$ as
\begin{align*}
\varphi(t)
=\varphi_0
+\int_{0}^{t}\partial\bar{\partial}^{*_{s}}
\omega(s) ds.
\end{align*}
So $\partial\omega(t)=-\bar{\partial}\varphi(t)$ and the only thing is to show $\partial\varphi(t)=0$.
It is easy to check by noticing $\bar{\partial}\varphi(0)=\bar{\partial}\varphi_0=0$.
Thus we complete our proof.
\end{proof}
\begin{remark}\label{Rm:pluriclosedFlowBCzero}
From the proof above we know that if $\varphi_0$ is $\partial$-exact, i.e. there exists a $(1,0)$-form $\alpha_0$ such that $\partial\alpha_0=\varphi_0$.
Then $\varphi(t)$ is $\partial$-exact and actually we can choose $\alpha(t)$ as
\begin{align*}
\alpha(t)
=\alpha_0
+\int_{0}^{t}\bar{\partial}^{*_{s}}\omega(s)ds.
\end{align*}
Meanwhile, we have $\partial\omega(t)=\partial\bar{\partial}\alpha(t)$.
In fact, it is the case when $[\partial\omega_0]=0\in\text{H}^{2,1}_{BC}(M,\mathbb{C})$.
Here $\text{H}^{2,1}_{BC}(M,\mathbb{C})$ is the Bott-Chern cohomology group.
\end{remark}

By an observation, we find that the evolution equation of $(2,0)$-part is exactly determined by the $(2,0)$-part of Bismut-Ricci form.
Thus, we can actually regard pluriclosed flow as a flow of Hermitian-symplectic forms by adding the extra evolution equation of $(2,0)$-part.

\begin{definition}
Let  $\Omega_0=\varphi_0+\omega_0+\bar{\varphi}_0$
be a Hermitian-symplectic form, we can define a flow with initial data $\Omega_0$ as
\begin{equation}
\label{Eq:H-SFlow}
\left\{
\begin{aligned}
&\frac{\partial}{\partial t}\omega
=\partial\partial^*\omega
+\bar{\partial}\bar{\partial}^*\omega
+\sqrt{-1}\partial\bar{\partial}\log\det g
\\
&\frac{\partial}{\partial t}\varphi
=-\partial\text{\rm tr}_{g}(\bar{\partial}\varphi)
\\
&\omega(0)=\omega_0,\quad \varphi(0)=\varphi_0.
\end{aligned}
\right.
\end{equation}
in which, $\text{\rm tr}_{g}(\beta)\triangleq\sqrt{-1}\Lambda_{\omega}\beta$ for arbitrary differential forms $\beta$.
And $\Lambda_{\omega}$ is the adjoint operator of Lefschetz operator $L_{\omega}:\beta\mapsto\omega\wedge\beta$
\end{definition}

We have a few things to say about this definition.
Recall that $\Omega_0$ is a Hermitian-symplectic form, so its (1,1)-part $\omega_0$ is pluriclosed and we denote $g_0$ the metric corresponding to $\omega_0$.
The first equation is just the pluriclosed flow with initial metric $\omega_0$, thus $\omega(t)$ exists uniquely for at least a short time.
So the evolution equation of $\varphi$ is well-defined when $g(t)$ exists.

Then we show $\varphi(t)$ exists uniquely in short time.
That is the lemma below.

\begin{lemma}\label{Lemma:(p,0)flow}
Given a smooth one-parameter family of metrics $g(t)$ in time interval $[0,T)$.
Consider the flow of (p,0)-forms $\alpha$ with initial data $\alpha_0$ satisfying $\partial\alpha_0=0$
\begin{equation*}
\left\{
\begin{aligned}
&\frac{\partial}{\partial t}\alpha
=-\partial\text{\rm tr}_{g}(\bar{\partial}\alpha)
\\
&\alpha(0)=\alpha_0
\end{aligned}
\right.
\end{equation*}
Then, there exists a unique $\partial$-closed solution $\alpha(t)$ in $[0,\varepsilon)$, in which $\varepsilon<T$ is a positive number depending on $\alpha_0$.
\end{lemma}
\begin{proof}
We denote $\mathcal{A}^{p,q}$ the sheaf of differential $(p,q)$-forms and $\mathcal{A}^{p,q}_{\partial}=\{\alpha\in\mathcal{A}^{p,q}|\partial\alpha=0\}$ the set of $\partial$-closed $(p,q)$-forms.
And notice the differential operator
\begin{align*}
\Phi=-\partial\text{tr}_{g}(\bar{\partial}\cdot)
\Big|_{\mathcal{A}_{\partial}^{p,0}}
:\mathcal{A}_{\partial}^{p,0}
\to
\mathcal{A}_{\partial}^{p,0}
\end{align*}
is a linear differential operator from $\mathcal{A}_{\partial}^{p,0}$ to itself.

We claim that $\Phi$ is actually elliptic.

In local complex coordinates, we assume that
\begin{align*}
\alpha
&=\sum_{i_1<\cdots<i_p}
\alpha_{i_1\cdots i_p}
dz^{i_1}\wedge\cdots\wedge dz^{i_p}
\\
&=\frac{1}{p!}
\alpha_{i_1\cdots i_p}
dz^{i_1}\wedge\cdots\wedge dz^{i_p}
\end{align*}
where the subscripts are antisymmetric.
By direct computation, we have
\begin{align*}
\Phi(\alpha)
&=-\partial\text{tr}_{g}(\bar{\partial}\alpha)
\\
&=-\partial\text{tr}_{g}
(
\frac{1}{p!}
\alpha_{i_1\cdots i_p,\bar{t}}
d\bar{z}^t\wedge
dz^{i_1}\wedge\cdots\wedge dz^{i_p}
)
\\
&=\frac{1}{(p-1)!}\partial
(
g^{\bar{t}s}\alpha_{si_{2}\cdots i_{p},\bar{t}}
dz^{i_2}\wedge\cdots\wedge dz^{i_p}
)
\\
&=\frac{1}{(p-1)!}
(g^{\bar{t}s}\alpha_{si_{2}\cdots i_{p},\bar{t}})_{,i_1}
dz^{i_1}\wedge\cdots\wedge dz^{i_p}
\end{align*}
The terms containing second order derivatives with respect to $\alpha$ are
\begin{align*}
\text{2nd order}
&=\frac{1}{(p-1)!}g^{\bar{t}s}
\alpha_{i_{1}\cdots i_{p-1}s,i_p\bar{t}}
dz^{i_1}\wedge\cdots\wedge dz^{i_p}
\\
&=\sum_{i_1<\cdots<i_p}
g^{\bar{t}s}
\sum_{a=1}^{p}(-1)^{p-a}
\alpha_{i_{1}\cdots\widehat{i_{a}}\cdots i_{p}s,i_a\bar{t}}
dz^{i_1}\wedge\cdots\wedge dz^{i_p}
\\
&=\sum_{i_1<\cdots<i_p}
g^{\bar{t}s}
\alpha_{i_{1}\cdots i_{p},s\bar{t}}
dz^{i_1}\wedge\cdots\wedge dz^{i_p}
\end{align*}
The last line uses the condition $\partial\alpha=0$, which is
\begin{align*}
\alpha_{i_{1}\cdots i_{p},s}
-\sum_{a=1}^{p}(-1)^{p-a}
\alpha_{i_{1}\cdots\widehat{i_a}\cdots i_{p}s,i_a}
=0
\end{align*}
in local coordinates.
So we prove $\Phi$ is elliptic.
Then, by the standard theorem of parabolic equations, we get the short-time existence and uniqueness of solutions.
\end{proof}
Applying Lemma \ref{Lemma:(p,0)flow} in the case of $p=2$ and note that $\varphi_0$ is $\partial$-closed for $\Omega_0$ is closed.
So we get the short-time existence and uniqueness of $\varphi(t)$.

To show the flow (\ref{Eq:H-SFlow}) preserves Hermitian-symplectic structures, we need the next lemma.
\begin{lemma}\label{Lemma:(p,q)Flow}
Given a smooth one-parameter family of metrics $g(t)$ in time interval $[0,T)$.
Consider the flow of $(p,q)$-form $\phi$ with initial data $\phi_0$ satisfying $d\phi_0=0$
\begin{equation*}
\left\{
\begin{aligned}
&\frac{\partial}{\partial t}\phi
=\partial\bar{\partial}
\text{\rm tr}_{g}(\phi)
\\
&\phi(0)=\phi_0
\end{aligned}
\right.
\end{equation*}
Then, there exists a unique closed solution $\phi(t)$ in $[0,\varepsilon)$, in which $\varepsilon<T$ is a positive number depending on $\phi_0$.
\end{lemma}
\begin{proof}
This proof is similar with Lemma \ref{Lemma:(p,0)flow}.
Denote $\mathcal{A}^{p,q}_{d}=\{\phi\in\mathcal{A}^{p,q}|d\phi=0\}$ the set of closed $(p,q)$-forms.
And consider the operator
\begin{align*}
\Psi=\partial\bar{\partial}\text{tr}_{g}(\cdot)
\Big|_{\mathcal{A}^{p,q}_{d}}
:\mathcal{A}^{p,q}_{d}\to\mathcal{A}^{p,q}_{d}
\end{align*}
from $\mathcal{A}^{p,q}_{d}$ to itself.

We claim that $\Psi$ is a linear elliptic operator.

In local complex coordinates, we assume
\begin{align*}
\phi
&=\sum_{\substack{i_1<\cdots<i_r\\j_1<\cdots<j_s}}
\phi_{i_1\cdots i_r\bar{j_1}\cdots\bar{j_s}}
dz^{i_1}\wedge\cdots\wedge dz^{i_r}
\wedge
d\bar{z}^{j_1}\wedge\cdots\wedge d\bar{z}^{j_s}
\\
&=\frac{1}{r!s!}
\phi_{i_1\cdots i_r\bar{j_1}\cdots\bar{j_s}}
dz^{i_1}\wedge\cdots\wedge dz^{i_r}
\wedge
d\bar{z}^{j_1}\wedge\cdots\wedge d\bar{z}^{j_s}
\end{align*}
where the subscripts are antisymmetric. 
Then
\begin{align*}
\text{tr}_{g}(\phi)
&=\frac{1}{(r-1)!(s-1)!}
g^{\bar{q}p}
\phi_{i_1\cdots i_{r-1}p\bar{q}\bar{j_2}\cdots\bar{j_s}}
dz^{i_1}\wedge\cdots\wedge dz^{i_{r-1}}
\wedge
d\bar{z}^{j_2}\wedge\cdots\wedge d\bar{z}^{j_s}
\end{align*}
And consider the second order term of $\partial\bar{\partial}\text{tr}_{g}(\phi)$
\begin{align*}
\text{2nd order}
&=\frac{1}{(r-1)!(s-1)!}
g^{\bar{q}p}
\phi_{i_1\cdots i_{r-1}p\bar{q}\bar{j_2}\cdots\bar{j_s},i_r\bar{j_1}}
dz^{i_1}\wedge\cdots\wedge dz^{i_{r}}
\wedge
d\bar{z}^{j_1}\wedge\cdots\wedge d\bar{z}^{j_s}
\\
&=\frac{1}{(s-1)!}
\sum_{i_1<\cdots<i_r}
g^{\bar{q}p}
\sum_{a=1}^{r}(-1)^{r-a}
\phi_{i_1\cdots\widehat{i_{a}}\cdots i_{r-1}p
\bar{q}\bar{j_2}\cdots\bar{j_s},i_{a}\bar{j_1}}
\\
&\cdot dz^{i_1}\wedge\cdots\wedge dz^{i_{r}}
\wedge
d\bar{z}^{j_1}\wedge\cdots\wedge d\bar{z}^{j_s}
\\
&=\frac{1}{(s-1)!}
\sum_{i_1<\cdots<i_r}
g^{\bar{q}p}
\phi_{i_1\cdots i_{r}
\bar{q}\bar{j_2}\cdots\bar{j_s},p\bar{j_1}}
dz^{i_1}\wedge\cdots\wedge dz^{i_{r}}
\wedge
d\bar{z}^{j_1}\wedge\cdots\wedge d\bar{z}^{j_s}
\end{align*}
The last line is because $\partial\phi=0$, which if and only if
\begin{align*}
\phi_{i_1\cdots i_{r}
\bar{j_1}\cdots\bar{j_s},p}
-\sum_{a=1}^{r}(-1)^{r-a}
\phi_{i_1\cdots\widehat{i_a}\cdots i_{r}p
\bar{j_1}\cdots\bar{j_s},i_a}
=0.
\end{align*}
Similarly, using the fact $\bar{\partial}\phi=0$, the terms of second order are
\begin{align*}
\text{2nd order}
&=\frac{1}{(s-1)!}\sum_{i_1<\cdots<i_r}
g^{\bar{q}p}
\phi_{i_1\cdots i_{r}
\bar{j_2}\cdots\bar{j_s}\bar{q},p\bar{j_1}}
dz^{i_1}\wedge\cdots\wedge dz^{i_{r}}
\wedge
d\bar{z}^{j_2}\wedge\cdots\wedge d\bar{z}^{j_s}
\wedge d\bar{z}^{j_1}
\\
&=\sum_{\substack{i_1<\cdots<i_r\\j_1<\cdots<j_s}}
g^{\bar{q}p}
\phi_{i_1\cdots i_{r}
\bar{j_1}\cdots\bar{j_s},p\bar{q}}
dz^{i_1}\wedge\cdots\wedge dz^{i_{r}}
\wedge
d\bar{z}^{j_1}\wedge\cdots\wedge d\bar{z}^{j_s}
\end{align*}
Thus $\Psi$ is elliptic and we complete our proof.
\end{proof}

Now we show that the Hermitian-symplectic structures are preserved by flow \eqref{Eq:H-SFlow}.
This is why we say it is a flow of Hermitian-symplectic forms.

\begin{theorem}\label{Thm:H-SFlowByPDE}
Hermitian-symplectic structures are preserved by flow \eqref{Eq:H-SFlow}.
\end{theorem}
\begin{proof}
Let $\Omega(t)=\varphi(t)+\omega(t)+\bar{\varphi}(t)$ be a solution to flow \eqref{Eq:H-SFlow} with initial data $\Omega_0=\varphi_0+\omega_0+\bar{\varphi}_0$.
Recall the definition of Hermitian-symplectic forms.
The only thing we need to check is that $\Omega(t)$ is closed, which equivalents to equations (\ref{Eq:existenceOfHS}).
Applying Lemma \ref{Lemma:(p,0)flow} in the case of $p=2$, we know that $\varphi(t)$ is $\partial$-closed.
To complete our proof, we need to show $\bar{\partial}\varphi(t)+\partial\omega(t)=0$.

Applying $\partial$ and $\bar{\partial}$ to both sides of the evolution equations of $\omega$ and $\varphi$, respectively, we get
\begin{align}
&\frac{\partial}{\partial t}\partial\omega
=\partial\bar{\partial}\bar{\partial}^*\omega
\label{Eq:torsionFlow-omega}
\\
&\frac{\partial}{\partial t}\bar{\partial}\varphi
=\partial\bar{\partial}\text{tr}_{g}(\bar{\partial}\varphi).
\label{Eq:torsionFlow-varphi}
\end{align}
Firstly, we show that both equations \eqref{Eq:torsionFlow-omega} and \eqref{Eq:torsionFlow-varphi} have the same form.
By direct computation, we have
\begin{align*}
\bar{\partial}^*\omega
&=-*\partial*\omega
\\
&=-\frac{1}{(n-1)!}*\partial(\omega^{n-1})
\\
&=-\frac{1}{(n-2)!}*(\omega^{n-2}\wedge\partial\omega)
\\
&=\text{tr}_{g}(\partial\omega)
\end{align*}
The last line uses Lemma \ref{Lemma:traceAndHodge} in the appendix, which can be proved by direct computation.
Thus equations \eqref{Eq:torsionFlow-omega} and \eqref{Eq:torsionFlow-varphi} both have the form in Lemma \ref{Lemma:(p,q)Flow}
\begin{align}
\frac{\partial}{\partial t}\partial\phi
=\partial\bar{\partial}\text{tr}_{g}(\phi)
\label{Eq:torsionFlow}
\end{align}
for closed $(2,1)$ forms $\phi$.

Note that both $\partial\omega(t)$ and $-\bar{\partial}\varphi(t)$ are solutions to \eqref{Eq:torsionFlow} with initial data $\partial\omega_0$ and $-\bar{\partial}\varphi_0$, respectively.
Meanwhile we have $\partial\omega_0=-\bar{\partial}\varphi_0$ for $\Omega_0$ is a Hermitian-symplectic form.
Applying Lemma \ref{Lemma:(p,q)Flow}, we get $\partial\omega(t)=-\bar{\partial}\varphi(t)$ by the uniqueness of solutions.
Thus we complete our proof. 
\end{proof}
\begin{remark}\label{Rm:varphiFlowTorsion}
From the above proof and Lemma \ref{Lemma:traceAndHodge}, we know that the evolution equation of $\varphi$ in \eqref{Eq:H-SFlow} can rewrite
as
\begin{align*}
\frac{\partial}{\partial t}\varphi
=\partial\bar{\partial}^*\omega.
\end{align*}
And by direct computation, we have $\partial\bar{\partial}^*\omega=-\rho^{2,0}(\omega)$.
The right-hand term is exactly the $(2,0)$-part of Bismut-Ricci form.
So we say the additional equation is natural with respect to pluriclosed flow.
\end{remark}

\section{A topological obstruction}

In this section, we will introduce a positive function $\mathcal{V}:\Omega\mapsto\mathcal{V}(\Omega)>0$ related to Hermitian-symplectic forms.
And show that when a Hermitian-symplectic form deform along the parabolic flow \eqref{Eq:H-SFlow}, the function $\mathcal{V}(t)\triangleq\mathcal{V}(\Omega(t))$ is actually a polynomial with respect to time $t$.
Meanwhile, all coefficients of this polynomial are constants depending only on initial data $\Omega_0$ and the underlying complex manifold.
That will bring some obstructions to how long a solution to flow \eqref{Eq:H-SFlow} exists.
Especially, it brings a topological obstruction to global solutions to pluriclosed flow with Hermitian-symplectic initial data.

\subsection{Definition of exponential-type volume function}

Let\rq{}s begin with the definition of $\mathcal{V}$.
Given a complex manifold $(M^{2n},J)$.
Denote $\mathcal{A}^{2}_{HS}(M)$ the set of Hermitian-symplectic forms on the manifold and we assume it is not empty.

\begin{definition}
We define the exponential-type volume function
$\mathcal{V}:\mathcal{A}^{2}_{HS}(M)\to (0,+\infty)$ by
\begin{align}
\mathcal{V}(\Omega)
=\sum_{k\geq0}\frac{1}{(k!)^2}
(\varphi^k,\varphi^k)_{\omega}
\label{Def:volumeTypeFunction}
\end{align}
Here $\Omega=\varphi+\omega+\bar{\varphi}$ and $(\cdot,\cdot)_{\omega}$ is the inner product induced by metric $\omega$.
\end{definition}

Here we will give some explanations about this definition.
Firstly, definition \eqref{Def:volumeTypeFunction} is actually a sum of finite term for a given manifold has finite dimension.
Secondly, notice that $\omega$ is the $(1,1)$-part of a Hermitian-symplectic form, so it is a metric and we can have the inner product $(\cdot,\cdot)_{\omega}$.
For this reason, $\mathcal{V}$ can not define for all symplectic forms.
And when $n=2$,
$\mathcal{V}(\Omega)
=(1,1)_{\omega}+(\varphi,\varphi)_{\omega}
=\frac{1}{2}(\Omega,\Omega)_{\omega}
$, is the volume of Hermitian-symplectic form up to a constant.
This is why we call $\mathcal{V}$ volume function.
\begin{remark}
If we define the formal exponential map for differential forms by
\begin{align*}
\exp(\phi)=
\sum_{k\geq0}\frac{\phi^k}{k!}
\end{align*}
Then $\mathcal{V}(\omega)=(\exp(\varphi),\exp(\varphi))_{\omega}$ for the inner product of two forms with different bidegree is zero.
This is why we call $\mathcal{V}$ exponential-type.
\end{remark}

\subsection{Exponential-type volume function along Hermitian-symplectic flow}

In this subsection, we study the exponential-type volume function $\mathcal{V}(t)=\mathcal{V}(\Omega(t))$ along the flow \eqref{Eq:H-SFlow}.
And we claim that $\mathcal{V}$ is actually a polynomial of $t$ and the coefficient of its highest degree term is a topological quantity of the manifold.

More precisely, we have
\begin{theorem}\label{Thm:Vt}
For a solution to \eqref{Eq:H-SFlow} with initial data $\Omega_0$, the function $\mathcal{V}(t)$ is a polynomial of degree at most $n$
\begin{align*}
\mathcal{V}(t)
=\sum_{i=0}^{n}a_i t^i
\end{align*}
with coefficients
\begin{align*}
a_i
=\frac{1}{i!}
\sum_{\substack{k\geq0\\2k+i\leq n}}
\frac{1}{(k!)^2(n\m 2k\m i)!}
\int_{M}
\varphi_0^k
\wedge\bar{\varphi}_0^k
\wedge\omega^{n-2k-i}_0
\wedge(\sqrt{-1}\dd\db\log\det g_0)^{i}
\end{align*}
Specially, we have
\begin{align*}
a_n=\frac{1}{n!}(-1)^n
\int_{M}
c_{1}(M)^{n}
=\frac{1}{n!}(-1)^n c_1^n
\end{align*}
in which $c_{1}(M)$ is the first Chern class of $(M^{2n},J)$.
\end{theorem}
We put the proof at the end of this section.

From Theorem \ref{Thm:Vt}, we know that $a_0=\mathcal{V}(0)=\mathcal{V}(\Omega_0)$ is a positive number.
Other coefficients may be difficult to calculate in general case.
The formula of $\mathcal{V}$ can bring us an obstruction to how long a solution to flow \eqref{Eq:H-SFlow} exists.

Formally, we have
\begin{corollary}
Assume a solution to flow \eqref{Eq:H-SFlow} exists on $[0,T)$, then the polynomial of $t$ defined by Theorem \ref{Thm:Vt} must have no roots on $[0,T)$.
\end{corollary}
\begin{proof}
Just note that the polynomial is the exponential-type function and it is positive by definition.
\end{proof}

For global solutions, i.e. solutions exist on $[0,+\infty)$, we have a more useful obstruction.
And we state it as a theorem.
\begin{theorem}\label{Thm:obstructionForH-SFlow}
If there exists a global solution to flow \eqref{Eq:H-SFlow} on $(M^{2n},J)$, then $(-1)^n c_{1}^{n}$ must be nonnegative.
\end{theorem}
\begin{proof}
If $(-1)^{n} c_{1}^{n}$ is negative, then $\mathcal{V}(t)$ must have a root in $[0,+\infty)$.
This is a contradiction to long-time existence of the solution.
\end{proof}

In Remark \ref{Rm:pluriclosedFlowBCzero}, we point out that pluriclosed flow with initial data $\omega_0$ satisfying $[\dd\omega_0]=0\in\text{H}^{2,1}_{BC}(M,\mathbb{C})$ is a special case that can be extended to a Hermitian-symplectic form.
So the sign of $(-1)^n c_{1}^{n}$ is also an obstruction to global solutions to pluriclosed flow under those assumption.

\subsection{Proof of Theorem \ref{Thm:Vt}}

In this subsection, we give the proof of Theorem \ref{Thm:Vt}.

Firstly, we need the following lemma.

\begin{lemma}\label{Lem:betaClosed}
Assume $\Omega(t)=\varphi(t)+\omega(t)+\bar{\varphi}(t)$ is a solution to \eqref{Eq:H-SFlow} with initial data $\Omega_0$.
Define the auxiliary differential forms
\begin{equation}
\label{Def:alpha}
\alpha[k,s](t)=
\left\{
\begin{aligned}
&\varphi^{k}(t)\wedge\bar{\varphi}^{k}(t)
\wedge\omega^{n-2k-s}(t)
,&&s,k\geq0\text{\rm\ and }2k+s\leq n
\\
&0,&&\text{otherwise}
\end{aligned}
\right.
\end{equation}
and
\begin{align}
\label{Def:beta}
\beta[s](t)
=\sum_{k\geq0}
\frac{1}{(k!)^2(n-2k-s)!}
\alpha[k,s](t)
\end{align}
Then $\beta[s](t)$ is pluriclosed.
\end{lemma}
\begin{proof}
By direct computation,
\begin{align*}
\partial\bar{\partial}\alpha[k,s]
&=
\partial
\left\{
k\bar{\partial}\varphi
\wedge\varphi^{k-1}
\wedge\bar{\varphi}^{k}
\wedge\omega^{n-2k-s}
+
(n\m 2k \m s)
\varphi^{k}
\wedge\bar{\varphi}^{k}
\wedge\omega^{n-2k-s-1}
\wedge\bar{\partial}\omega
\right\}
\\
&=-k^2\bar{\partial}\varphi
\wedge\varphi^{k-1}
\wedge\bar{\varphi}^{k-1}
\wedge\partial\bar{\varphi}
\wedge\omega^{n-2k-s}
\\
&\quad-
k(n\m 2k\m s)
\bar{\partial}\varphi
\wedge\varphi^{k-1}
\wedge\bar{\varphi}^{k}
\wedge\partial\omega
\wedge\omega^{n-2k-s-1}
\\
&\quad+
k(n\m 2k\m s)
\partial\varphi
\wedge\varphi^{k-1}
\wedge\bar{\varphi}^{k}
\wedge\bar{\partial}\omega
\wedge\omega^{n-2k-s-1}
\\
&\quad+
(n\m 2k\m s)(n\m 2k\m s\m 1)
\varphi^{k}
\wedge\bar{\varphi}^{k}
\wedge\omega^{n-2k-s-2}
\wedge\partial\omega
\wedge\bar{\partial}\omega
\\
&=\Big\{
-k^2\alpha[k\m 1,s\p 2]
+(n\m 2k\m s)(n\m 2k\m s\m 1)\alpha[k,s\p 2]
\Big\}
\wedge
\partial\omega
\wedge\bar{\partial}\omega
\end{align*}

Note that $\partial\omega=-\bar{\partial}\varphi$.
So
$\bar{\partial}\varphi\wedge\partial\omega
=-\partial\omega\wedge\partial\omega=0$
for $\partial\omega$ is an odd form.
Thus the third and fourth lines equal zero.
By direct computation, it is easy to check that the formula above is also true when $k=0$ and $k=(n-s)/2$.

Then we have
\begin{align*}
\partial\bar{\partial}\beta[s]
&=
\sum_{k\geq0}
\frac{1}{(k!)^2(n\m 2k\m s)!}
\Big(
-k^2\alpha[k\m 1,s\p 2]
+(n\m 2k\m s)
(n\m 2k\m s\m 1)\alpha[k,s\p 2]
\Big)
\\
&\qquad\wedge
\partial\omega
\wedge\bar{\partial}\omega
\\
&=
-\sum_{k\geq1}
\frac{\alpha[k\m 1,s\p 2]}{(k-1!)^2(n\m 2k\m s)!}
+\sum_{k\geq0}
\frac{\alpha[k,s\p 2]}{(k!)^2(n\m 2k\m s\m 2)!}
\\
&=
-\sum_{k\geq0}
\frac{\alpha[k,s\p 2]}{(k!)^2(n\m 2k\m s\m 2)!}
+\sum_{k\geq0}
\frac{\alpha[k,s\p 2]}{(k!)^2(n\m 2k\m s\m 2)!}
\\
&=0
\end{align*}
The second equal sign is because $\alpha[-1,s]=0$ and we just replace the index $k-1$ by $k$ in the third line.
Thus we complete our proof.
\end{proof}

To compute the derivative of function $\mathcal{V}(t)=\mathcal{V}(\Omega(t))$ with respect to time $t$.
We need the next two lemmas.

\begin{lemma}\label{Lem:d1P}
Given a fixed real form $\phi$ with degree $2s$ which is $\partial$ and $\bar{\partial}$ closed.
Define a function with respect to $t$ by
\begin{align*}
P[k,s;\phi](t)
=\int_{M}\alpha[k,s](t)\wedge\phi
\end{align*}
in which $\alpha$ is defined by \eqref{Def:alpha}.
Then we have
\begin{equation}
\label{Eq:d1P}
\begin{aligned}
\frac{d}{dt}P[k,s;\phi]
&=
-k^2(A[k,s]+\overline{A[k,s]})
\\
&\qquad+
(n\m 2k\m s)(n\m 2k\m s\m 1)
(B[k,s]+\overline{B[k,s]})
\\
&\qquad+
(n\m 2k\m s)
\int_{M}
\alpha[k,s\p 1]
\wedge\sqrt{-1}\dd\db\log\det g
\wedge\phi
\end{aligned}
\end{equation}
where
\begin{align}
\label{Eq:d1Ptermbefore}
A[k,s]
=\int_{M}
\db^*\omega
\wedge\alpha[k\m 1,s\p 2]
\wedge\db\omega
\wedge\phi
\end{align}
and
\begin{align}
\label{Eq:d1Ptermafter}
B[k,s]
=\int_{M}
\db^*\omega
\wedge\alpha[k,s\p 2]
\wedge\db\omega
\wedge\phi.
\end{align}
\end{lemma}

\begin{proof}
Here we use a point overhead to represent the derivative with respect to time $t$.
Derivate $P$ with respect to $t$ and notice that $\omega(t)$ satisfies equations \eqref{Eq:torsionFlow-omega}.
\begin{align*}
\frac{d}{dt}P[k,s;\phi]
&=
\frac{d}{dt}
\int_{M}
\varphi^{k}\wedge\bar{\varphi}^{k}
\wedge\omega^{n-2k-s}
\wedge\phi
\\
&=
k
\left\{
\int_{M}
\dot{\varphi}\wedge
\varphi^{k-1}\wedge\bar{\varphi}^{k}
\wedge\omega^{n-2k-s}
\wedge\phi
+
\int_{M}
\varphi^{k}\wedge
\dot{\bar{\varphi}}\wedge
\bar{\varphi}^{k-1}
\wedge\omega^{n-2k-s}
\wedge\phi
\right\}
\\
&\quad+
(n\m 2k\m s)
\int_{M}
\varphi^{k}
\wedge\bar{\varphi}^{k}
\wedge\omega^{n-2k-s-1}
\wedge\bar{\partial}\bar{\partial}^*\omega
\wedge\phi
\\
&\quad+
(n\m 2k\m s)
\int_{M}
\varphi^{k}\wedge
\bar{\varphi}^{k}
\wedge\omega^{n-2k-s-1}
\wedge\partial\partial^*\omega
\wedge\phi
\\
&\quad+
(n\m 2k\m s)
\int_{M}
\varphi^{k}\wedge
\bar{\varphi}^{k}
\wedge\omega^{n-2k-s-1}
\wedge\sqrt{-1}\partial\bar{\partial}\log\det g
\wedge\phi
\end{align*}
The second and the fourth terms are just the conjugation of the first and the third terms, respectively.
Recall Remark \ref{Rm:varphiFlowTorsion} and note $\phi$ is $\partial$-closed. 
Then the first term is 
\begin{align*}
\text{1st term }
&=
k\int_{M}
\partial\bar{\partial}^*\omega
\wedge\varphi^{k-1}
\wedge\bar{\varphi}^{k}
\wedge\omega^{n-2k-s}
\wedge\phi
\\
&=
k^2
\int_{M}
\bar{\partial}^*\omega
\wedge\varphi^{k-1}
\wedge\partial\bar{\varphi}
\wedge\bar{\varphi}^{k-1}
\wedge\omega^{n-2k-s}
\wedge\phi
\\
&\qquad+
k(n\m 2k\m s)
\int_{M}
\bar{\partial}^*\omega
\wedge\varphi^{k-1}
\wedge\bar{\varphi}^{k}
\wedge\partial\omega
\wedge\omega^{n-2k-s-1}
\wedge\phi
\end{align*}
Similarly,
\begin{align*}
\text{3rd term }
&=
(n\m 2k\m s)
\int_{M}
\bar{\partial}\bar{\partial}^*\omega
\wedge\varphi^{k}
\wedge\bar{\varphi}^{k}
\wedge\omega^{n-2k-s-1}
\wedge\phi
\\
&=
(n\m 2k\m s)(n\m 2k\m s\m 1)
\int_{M}
\bar{\partial}^*\omega
\wedge\varphi^{k}
\wedge\bar{\varphi}^{k}
\wedge\bar{\partial}\omega
\wedge\omega^{n-2k-s-2}
\wedge\phi
\\
&\qquad+
k(n\m 2k\m s)
\int_{M}
\bar{\partial}^*\omega
\wedge\bar{\partial}\varphi
\wedge\varphi^{k-1}
\wedge\bar{\varphi}^{k}
\wedge\omega^{n-2k-s-1}
\wedge\phi
\end{align*}
Here we use $\bar{\partial}\phi=0$.
The sum in last line of \lq\lq{}1st term\rq\rq{} and in last line of \lq\lq{}3rd term\rq\rq{} are zero for $\dd\omega+\db\varphi=0$.

Summing the above and their conjugations, we get formula \eqref{Eq:d1P}.
\end{proof}
Notice that we can also choose $s=0$ in the above lemma.
And now we give another lemma we need.
\begin{lemma}\label{Lem:derivativeOfQ}
Given a fixed real form $\phi$ with degree $2s$ which is $\partial$ and $\bar{\partial}$ closed.
Define a function with respect to $t$ by
\begin{align}
\label{Def:Q}
Q[s;\phi](t)
=\int_{M}\beta[s](t)
\wedge\phi
\end{align}
in which $\beta$ is defined by \eqref{Def:beta}.
Then we have
\begin{align*}
\frac{d}{dt}Q[s;\phi](t)
&=Q[s+1;\sqrt{-1}\dd\db\log\det g(t)\wedge\phi]
\\
&=
\int_{M}\beta[s+1](t)
\wedge\sqrt{-1}\dd\db\log\det g(t)
\wedge\phi.
\end{align*}
\end{lemma}

\begin{proof}
Recall the definition of $\beta$ and use Lemma \ref{Lem:d1P}.
\begin{align*}
\frac{d}{dt}Q[s;\phi]
&=
\frac{d}{dt}
\int_{M}
\sum_{k\geq0}
\frac{1}{(k!)^2(n\m 2k\m s)!}
\alpha[k,s](t)
\wedge\phi
\\
&=
\sum_{k\geq0}
\frac{1}{(k!)^2(n\m 2k\m s)!}
\frac{d}{dt}
P[k,s;\phi]
\\
&=
\sum_{k\geq0}
\frac{1}{(k!)^2(n\m 2k\m s)!}
\Bigg\{
-k^2(A[k,s]+\overline{A[k,s]})
\\
&\qquad+
(n\m 2k\m s)(n\m 2k\m s\m 1)
(B[k,s]+\overline{B[k,s]})
\Bigg\}
\\
&\qquad+
\sum_{k\geq0}
\frac{n\m 2k\m s}{(k!)^2(n\m 2k\m s)!}
\int_{M}
\alpha[k,s\p 1]
\wedge\sqrt{-1}\dd\db\log\det g
\wedge\phi
\end{align*}
Here $A[k,s]$ and $B[k,s]$ are defined by \eqref{Eq:d1Ptermbefore} and \eqref{Eq:d1Ptermafter}, respectively.

Consider the term
\begin{align*}
&\sum_{k\geq0}
\frac{1}{(k!)^2(n\m 2k\m s)!}
\Big(-k^2A[k,s]+(n\m 2k\m s)(n\m 2k\m s\m 1)
B[k,s]
\Big)
\\
=&
-\sum_{k\geq1}
\frac{1}{((k\m 1)!)^2(n\m 2k\m s)!}
\int_{M}
\db^*\omega
\wedge\alpha[k\m 1,s\p 2]
\wedge\db\omega
\wedge\phi
\\
&\qquad+\sum_{k\geq0}
\frac{1}{(k!)^2(n\m 2k\m s\m 2)!}
\int_{M}
\db^*\omega
\wedge\alpha[k,s\p 2]
\wedge\db\omega
\wedge\phi
\\
=&
-\sum_{k\geq0}
\frac{1}{(k!)^2(n\m 2k\m s\m 2)!}
\int_{M}
\db^*\omega
\wedge\alpha[k,s\p 2]
\wedge\db\omega
\wedge\phi
\\
&\qquad+
\int_{M}
\db^*\omega
\wedge\beta[s\p 2]
\wedge\db\omega
\wedge\phi
\\
=&
-
\int_{M}
\db^*\omega
\wedge\beta[s\p 2]
\wedge\db\omega
\wedge\phi
+
\int_{M}
\db^*\omega
\wedge\beta[s\p 2]
\wedge\db\omega
\wedge\phi
\\
=&0
\end{align*}
The first equal sign is because $\alpha[-1,s+2]=0$.
\\
Thus,
\begin{align*}
\frac{d}{dt}Q[s;\phi]
&=
\sum_{k\geq0}
\frac{1}{(k!)^2(n\m 2k\m s\m 1)!}
\int_{M}
\alpha[k,s\p 1]
\wedge\sqrt{-1}\dd\db\log\det g
\wedge\phi
\\
&=
\int_{M}
\beta[s\p 1]
\wedge\sqrt{-1}\dd\db\log\det g
\wedge\phi
\\
&=
Q[s+1;\sqrt{-1}\dd\db\log\det g\wedge\phi].
\end{align*}
So we complete our proof.
\end{proof}

Now we are in the position to prove Theorem \ref{Thm:Vt}.

\begin{proof}[Proof of Theorem \ref{Thm:Vt}]
Assume $\Omega(t)=\varphi(t)\p \omega(t)\p \bar{\varphi}(t)$ is a solution to flow \eqref{Eq:H-SFlow}.
We denote $(\cdot,\cdot)_t$ the inner product induced by $\omega(t)$.
Notice that $\varphi^k$ is a prime form with respect to the adjoint Lefschetz operator defined by $\omega(t)$, for it is a $(2k,0)$-form.
Then by the Lefschetz theory,
\begin{align*}
(\varphi^k,\varphi^k)_t
&=\int_{M}\varphi^k
\wedge*\bar{\varphi}^k
=
\frac{1}{(n-2k)!}
\int_{M}\varphi^k
\wedge\bar{\varphi}^k
\wedge\omega^{n-2k}
\end{align*}
So we have
\begin{align*}
\mathcal{V}(t)
&=\sum_{k\geq0}
\frac{1}{(k!)^2}
(\varphi^k,\varphi^k)_t
\\
&=\sum_{k\geq0}
\frac{1}{(k!)^2(n\m 2k)!}
\int_{M}\varphi^k
\wedge\bar{\varphi}^k
\wedge\omega^{n-2k}
\\
&=
\int_{M}\beta[0]
\\
&=
Q[0;1](t).
\end{align*}
in which $\beta[s]$ and $Q[s;\phi]$ are defined by \eqref{Def:beta} and \eqref{Def:Q}, respectively.

Here we choose $\phi=1$, which is a real form of degree $0$ and satisfies $\dd 1=0$ and $\db 1=0$.
Thus Lemma \ref{Lem:derivativeOfQ} can be applied.

Derivate $\mathcal{V}(t)$ with respect to $t$ and apply Lemma \ref{Lem:derivativeOfQ}.
\begin{align*}
\frac{d}{dt}
\mathcal{V}(t)
&=
\frac{d}{dt}
Q[0;1](t)
\\
&=
Q[1;\sqrt{-1}\dd\db\log\det g](t)
\\
&=
\int_{M}
\beta[1](t)
\wedge\sqrt{-1}\dd\db\log\det g(t)
\\
&=
\int_{M}
\beta[1](t)
\wedge\dd\db\log\det h
+
\int_{M}
\beta[1](t)
\wedge\sqrt{-1}\dd\db\log\frac{\det g(t)}{\det h}
\\
&=Q[1;\sqrt{-1}\dd\db\log\det h]
\end{align*}
where $h$ is an arbitrary fixed Hermitian metric.
The last equal sign is because $\log\frac{\det g}{\det h}$ is a well-defined function on manifolds $M$.
Then applying Lemma \ref{Lem:betaClosed} we know that $\beta[1](t)$ is pluriclosed.
So by integrating by parts, we have
\begin{align*}
\int_{M}
\beta[1](t)
\wedge\sqrt{-1}\dd\db\log\frac{\det g(t)}{\det h}
=
\int_{M}
\dd\db\beta[1](t)
\wedge\sqrt{-1}\log\frac{\det g(t)}{\det h}
=0
\end{align*}

We complete the proof by induction.
The case of first order derivative has been proved.
Assume it is true for the $m$-th order derivative.
That is
\begin{align*}
\frac{d^m}{dt^m}
\mathcal{V}(t)
&=
\sum_{k\geq0}
\frac{1}{(k!)^2(n\m 2k\m m)!}
\int_{M}
\varphi^k
\wedge\bar{\varphi}^k
\wedge\omega^{n-2k-m}
\wedge(\sqrt{-1}\dd\db\log\det h)^{m}
\\
&=
\int_{M}
\beta[m]
\wedge(\sqrt{-1}\dd\db\log\det h)^{m}
\\
&=
Q[m;(\sqrt{-1}\dd\db\log\det h)^{m}]
\end{align*}

Then derivate it with respect to $t$ and apply Lemma \ref{Lem:betaClosed} and Lemma \ref{Lem:derivativeOfQ} again.
\begin{align*}
\frac{d^{m+1}}{dt^{m+1}}
\mathcal{V}(t)
&=
\frac{d}{dt}
Q[m;(\sqrt{-1}\dd\db\log\det h)^{m}]
\\
&=
Q[m\p 1;\sqrt{-1}\dd\db\log\det g(t)
\wedge(\sqrt{-1}\dd\db\log\det h)^{m}]
\\
&=
\int_{M}
\beta[m\p 1]
\wedge(\sqrt{-1}\dd\db\log\det h)^{m}
\wedge\sqrt{-1}\dd\db\log\det g(t)
\\
&
=\int_{M}
\beta[m\p 1]
\wedge(\sqrt{-1}\dd\db\log\det h)^{m+1}
\\
&=
\sum_{k\geq0}
\frac{1}{(k!)^2(n\m 2k\m m\m 1)!}
\int_{M}
\varphi^k
\mathord{\wedge}\bar{\varphi}^k
\mathord{\wedge}\omega^{n-2k-m-1}
\mathord{\wedge}(\sqrt{-1}\dd\db\log\det h)^{m+1}
\end{align*}
Notice that $\beta[m]=0$ for $m\geq n+1$, so the $m$-th order derivative of $\mathcal{V}$ equals zero for all $m\geq n+1$.
Specially,
when $m=n$, the dimension of manifold, we have
\begin{align*}
\frac{d^n}{dt^n}\mathcal{V}(t)
&=
Q[n,(\sqrt{-1}\dd\db\log\det h)^{n}]
\\
&=\int_{M}
(\sqrt{-1}\dd\db\log\det h)^{n}
\\
&=(-1)^{n}\int_{M}
c_{1}(M)^{n}
\\
&=(-1)^{n}c_{1}^{n},
\end{align*}
in which $c_{1}(M)$ is the first Chern class of $M$.

Thus $\mathcal{V}(t)=\sum a_it^i$ is a polynomial of degree at most $n$ and coefficients can be represented by
\begin{align*}
a_i
=\frac{1}{n!}\sum_{k\geq0}
\frac{1}{(k!)^2(n\m 2k\m i)!}
\int_{M}
\varphi_0^k
\wedge\bar{\varphi}_0^k
\wedge\omega^{n-2k-i}_0
\wedge(\sqrt{-1}\dd\db\log\det g_0)^{i}
\end{align*} 
Here we choose $h=g_0$, the initial metric.
So $a_i$ depends only on the complex manifold and initial data $\Omega_0$.
Specially, $a_n=(-1)^nc_1^n/n!$ is a topological quantity of the complex manifolds.
\end{proof}

\section{More discussion}

In this section we do some discussion in the case of compact complex surfaces.
One of the most important motivation of pluriclosed flow is to classify Kodaira\rq{}s class VII surfaces.
Inspired by Perelman\rq{}s great work on Ricci flow (see \cite{Pei1, Pei2, Pei3}), Tian and Streets use pluriclosed flow to find canonical metrics on non-K{\"a}hler manifolds in order to understand its geometry and topology.
A natural possible canonical metric is static metric, which satisfies $\omega=\lambda\rho^{1,1}$ for some real number $\lambda$.
Notice that the static metric with $\lambda\neq0$ is also Hermitian-symplectic.
And this is our motivation to study Hermitian-symplectic structures.
An important first step is that Tian and Streets\cite{PF} prove class $\text{VII}^{+}$ surfaces, (i.e. class VII surfaces with $b_2>0$), admit no static metrics.
But it is not true for class VII surfaces with $b_2=0$ because the standard metric of Hopf surface is static with $\lambda=0$.
Recall that Bogomolov\cite{Bogomolov1, Bogomolov2} proves a class VII surface with $b_2=0$ is biholomorphic to either a Hopf surface or an Inoue surface.

Tian and Streets\cite{PF} have shown that a compact complex surface admits Hermitian-symplectic structures must be K{\"a}hler using the classification of compact complex surfaces.
As a corollary, non-K{\"a}hler surfaces admit no static metric with $\lambda\neq0$.
Here we can give more information about the relationship between Hermitian-symplectic structures and K{\"a}hler structures in dimension two using flow \eqref{Eq:H-SFlow}.

We begin with an observation.

\begin{lemma}\label{Lem:monotoneDimension2}
If $\Omega(t)=\varphi(t)+\omega(t)+\bar{\varphi}(t)$ is a solution to \eqref{Eq:H-SFlow} with initial data $\Omega_0$ on a compact complex surface.
Then the function $F(t)\triangleq(\varphi(t),\varphi(t))_{\omega(t)}$ is monotonically decreasing.
Moreover, if the initial data $\omega_0$ is non-K{\"a}hler, it is strictly monotonically decreasing.
\end{lemma}
\begin{proof}
Consider the function $F(t)=(\varphi(t),\varphi(t))_{\omega(t)}$.
Derivate it with respect to $t$ and notice that $\varphi$ satisfies \eqref{Eq:H-SFlow}.
\begin{align*}
\frac{d}{dt}F
&=
\frac{d}{dt}
\int_{M}\varphi
\wedge\bar{\varphi}
\\
&=
\int_{M}
\frac{\dd}{\dd t}\varphi
\wedge\bar{\varphi}
+
\text{conjugation}
\\
&=
\int_{M}
\dd\db^*\omega
\wedge\bar{\varphi}
+
\text{conjugation}
\\
\end{align*}
The first line uses the fact $\varphi$ is a prime form with respect to the adjoint Lefschetz operator induced by $\omega(t)$.
Using integration by parts and noting that $\dd\omega+\db\varphi=0$, we have
\begin{align*}
\int_{M}
\dd\db^*\omega
\wedge\bar{\varphi}
=&
-\int_{M}
\db^*\omega
\wedge\db\omega
=
-(\db\omega,\db\omega)_t
\end{align*}
So
\begin{align*}
\frac{d}{dt}F
=-2(\db\omega,\db\omega)_t
\leq0.
\end{align*}
Thus $(\varphi,\varphi)_t$ is monotonically decreasing.

Then we prove it is actually strictly monotonically decreasing when $\omega_0$ is non-K{\"a}hler.
We just need to show that the solution $\omega(t)$ to pluriclosed flow is non-K{\"a}hler when initial metric is non-K{\"a}hler.

Assume $\tilde{\omega}_t$ is a solution with non-K{\"a}hler initial metric.
If $\dd\tilde{\omega}_{t_0}=0$ for some $t_0>0$, then we can find $t_{*}>0$ such that $\dd\omega_{t_{*}}=0$ and $\dd\tilde{\omega}_{t}\neq0$ for $t<t_{*}$.
From the standard theorem of parabolic equations, we get the uniqueness of solutions to flow \eqref{Eq:pluriclosedFlowHodge} with condition $\omega(t_{*})=\tilde{\omega}_{t_{*}}$ in small neighborhood $(t_{*}-\varepsilon,t_{*}+\varepsilon)$.
Tian and Streets\cite{PF} point out that pluriclosed flow degenerates to K{\"a}hler-Ricci flow when initial data is K{\"a}hler and so preserves K{\"a}hler condition.
Then $\dd\tilde{\omega}_t=0$ in $(t_{*}-\varepsilon,t_{*}+\varepsilon)$ by uniqueness.
This is a contradiction.
Thus $\tilde{\omega}_t$ will not be K{\"a}hler at any finite time when the solution exists.
\end{proof}

As an application, we show that a Hermitian-symplectic structure can deform to a K{\"a}hler structure, if \eqref{Eq:H-SFlow} converges at infinity time.
Notice that
\begin{align*}
(\varphi(t),\varphi(t))_{t}-(\varphi(0),\varphi(0))_{0}
=-\int_{0}^{t}(\dd\omega(s),\dd\omega(s))_{s}ds
\end{align*}
for all $t>0$.
So we have $\int_{0}^{+\infty}(\dd\omega(s),\dd\omega(s))_{s}ds<\infty$ for $(\varphi(t),\varphi(t))_{t}$ is always positive.
If $\omega(t)$ converges to $\omega_{\infty}$ at infinity time, then $(\dd\omega_\infty,\dd\omega_\infty)_{\omega_{\infty}}=0$.
Thus the limitation metric is K{\"a}hler.
In fact, this is true in every dimension.
And the proof is due to Jeffrey Streets.
We just need to make a slight adjustment to the proof of Proposition 4.2 in \cite{StreetsYury}.
In short, the limitation of a pluriclosed flow must be a steady soliton, which satisfies $d^{*}(e^{-f}H_{\infty})=0$ for some function $f$.
Here $H_{\infty}=\sqrt{-1}(\dd-\db)\omega_{\infty}$ is the torsion form of $\omega_{\infty}$.
Then we have $\db^{*}(e^{-f}\dd\omega_{\infty})=0$.
And note that $\omega_{\infty}$ is Hermitian-symplectic.
So
\begin{align*}
(e^{-f}\dd\omega_{\infty},\dd\omega_{\infty})_{\infty}
=(e^{-f}\dd\omega_{\infty},-\db\varphi_{\infty})_{\infty}
=(\db^{*}(e^{-f}\dd\omega_{\infty}),-\varphi_{\infty})_{\infty}
=0
\end{align*}
Thus $\dd\omega_{\infty}=0$ and the limitation is K{\"a}hler.

But not all solutions can exist for a long time.
Let\rq{}s consider the function $\mathcal{V}(t)$.
In dimension two, it is just the volume of Hermitian-symplectic forms $(\Omega(t),\Omega(t))_{\omega(t)}$ up to a constant.
Applying Theorem \ref{Thm:Vt}, we know that $\mathcal{V}(t)$ is a polynomial of degree at most $2$ such that all coefficients are determined by initial data.
So we can calculate all roots of $\mathcal{V}(t)$ explicitly.
By direct computation, we know that the coefficients are $a_0=\frac{1}{2}(\Omega_0,\Omega_0)_0>0$, $a_2=c_1^2/2$ and
\begin{align*}
a_1=\int_{M}\omega_0
\wedge\sqrt{-1}\dd\db\log\det g_0.
\end{align*}
All possibilities are listed in Table \ref{Tab:dimension2}.

\begin{table}[!htbp]
\centering
\caption{All possibilities in dimension two}
\renewcommand{\arraystretch}{1.5}
\renewcommand{\tabcolsep}{0.2cm}
\begin{tabular}{|c|c|c|c|c|}
\hline
$a_2$ &$a_1$ &$\Delta=a_1^2-4a_2a_0$ &Minimum positive root &Is an obstruction?
\\
\hline
$=0$ &$\geq 0$ &$-$ &$-$ &$-$
\\
\hline
$=0$ &$<0$ &$-$ &$-a_0/a_1$ &$\surd$
\\
\hline
$>0$ &$\geq0$ &$-$ &$-$ &$-$
\\
\hline
$>0$ &$<0$ &$< 0$ &$-$ &$-$
\\
\hline
$>0$ &$<0$ &$\geq 0$ &$-(a_1+\Delta)/(2a_2)$ &$\surd$
\\
\hline
$<0$ &$-$ &$-$ &$-(a_1+\Delta)/(2a_2)$ &$\surd$
\\
\hline
\end{tabular}
\label{Tab:dimension2}
\end{table}

For example, by the classification of Kodaira, the ruled surfaces of genus $f\geq 1$ (may see Table 10 in chapter {\rm VI} of \cite{MR2030225}, where uses $g$ to denote genus) have $c_1^2=8(1-f)$.
Thus there is no global solution to Hermitian-symplectic flow \eqref{Eq:H-SFlow} on ruled surfaces of genus great that $1$.

As the end of this section, we prove again that non-K{\"a}hler surfaces admit no static metric with $\lambda<0$ without help of classification of compact complex surfaces.

\begin{theorem}
Non-K{\"a}hler compact complex surfaces admit no static metric with $\lambda<0$.
\end{theorem}
\begin{proof}
Without loss of generality, we only need to consider the case of $\lambda=-1$.

Assume $\omega_0$ is a static metric such that $\omega_0=-\rho^{1,1}(\omega_0)$.
Notice that $\omega_0$ can be extended to a Hermitian-symplectic form by selecting $\varphi_0=-\rho^{2,0}(\omega_0)$.
And now $\omega(t)\triangleq(1+t)\omega_0$ is a global solution to flow \eqref{Eq:H-SFlow} because $\rho^{1,1}$ is a homogeneous operator with respect to metric, i.e. $\rho^{1,1}(C\omega)=\rho^{1,1}(\omega)$ for any constant $C$.
In this situation, we have $\varphi(t)=(1+t)\varphi_0$ and
\begin{align*}
F(t)=(\varphi(t),\varphi(t))_{\omega(t)}
=\int_{M}\varphi(t)\wedge\bar{\varphi}(t)
=(1+t)^{2}(\varphi_0,\varphi_0)_{\omega_0}.
\end{align*}
From Lemma \ref{Lem:monotoneDimension2}, we know that $F(t)$ is monotonically decreasing.
So we must have $(\varphi_0,\varphi_0)_{\omega_0}=0$.
That means $\varphi_0=0$ and $\omega_0$ is a K{\"a}hler metric.
Thus we complete our proof.
\end{proof}

\section{Appendix}

We give a proof of the lemma used above.
\begin{lemma}\label{Lemma:traceAndHodge}
Given a Hermitian manifold $(M^{2n},g)$ and denote $\omega$ the fundamental form corresponding to metric $g$.
Then we have
\begin{align*}
*(\omega^{n-2}\wedge\beta)
=-(n-2)!\text{\rm tr}_{g}(\beta)
\end{align*}
for arbitrary (2,1)-forms $\beta$.
\end{lemma}
\begin{proof}
Because we can do all calculation in the tangent space at a fixed point, so we choose an orthonormal basis so that 
\begin{align*}
\omega
=\sqrt{-1}\sum_{p}dz^p\wedge d\bar{z}^p
\end{align*}
And the volume form is
\begin{align*}
\frac{1}{n!}\omega^n
=(\sqrt{-1})^n
dz^1\wedge d\bar{z}^1
\wedge\cdots\wedge
dz^n\wedge d\bar{z}^n
\end{align*}
Recall that actions of the Hodge operator on $(n,n-1)$-forms are
\begin{align*}
*(
dz^1\wedge d\bar{z}^1
\wedge\cdots\wedge
dz^s\wedge\widehat{d\bar{z}^s}
\wedge\cdots\wedge
dz^n\wedge d\bar{z}^n
)
=(-\sqrt{-1})^n dz^s
\end{align*}
Assume
$\beta=\frac{1}{2}\beta_{ij\bar{k}}dz^i\wedge dz^j \wedge d\bar{z}^k$ such that
$\beta_{ij\bar{k}}+\beta_{ji\bar{k}}=0$.

By direct calculation,
\begin{align*}
\omega^{n-2}\wedge\beta
&=(\sqrt{-1})^{n-2}(n-2)!
(\sum_{s<t}
dz^1
\wedge\cdots\wedge
dz^s\wedge\widehat{d\bar{z}^s}
\wedge\cdots\wedge
dz^s\wedge\widehat{d\bar{z}^s}
\wedge\cdots\wedge
d\bar{z}^n
)
\\
&\qquad\qquad\wedge(\frac{1}{2}\beta_{ij\bar{k}}dz^i\wedge dz^j \wedge d\bar{z}^k)
\\
&=(\sqrt{-1})^{n-2}(n-2)!
\sum_{k}
\beta_{sk\bar{k}}
dz^1
\wedge\cdots\wedge
dz^s\wedge\widehat{d\bar{z}^s}
\wedge\cdots\wedge
d\bar{z}^n
\end{align*}
So
\begin{align*}
*(\omega^{n-2}\wedge\beta)
&=-(n-2)!\sum_{k}
\beta_{sk\bar{k}}dz^s
\\
&=-(n-2)!
\text{tr}_{g}(\beta)
\end{align*}
The last equal uses the definition of $\text{tr}_{g}(\cdot)$ and note that we calculate under orthonormal basis.
\end{proof}

\bibliographystyle{plain}
\bibliography{pcfhsYe.bib}

\end{document}